\documentclass{amsart}[12pt, papera4]
\usepackage{tikz}
\usepackage{amssymb}
\usepackage{graphicx}
\usepackage{amsfonts}
\usepackage[all,2cell]{xy} \UseAllTwocells \SilentMatrices
\usepackage{graphicx}
\usepackage{amsthm}
\usepackage{amstext}
\usepackage{amsmath}
\usepackage{amscd}
\usepackage[mathscr]{eucal}
\usepackage{url}

\newtheorem*{theorem*}{Theorem}
\newtheorem{theorem}{Theorem}[section]
\newtheorem{proposition}[theorem]{Proposition}

\newtheorem{lemma}[theorem]{Lemma}

\theoremstyle{definition}
\newtheorem{definition}[theorem]{Definition}
\newtheorem{example}[theorem]{Example}

\theoremstyle{remark}

\numberwithin{equation}{section}

\def\N{{\mathbb N}}

\def\F{{\mathscr{F}}}
\def\G{{\mathscr{G}}}
\def\FS{{\mathscr F_s}}

\begin{document}

	\centerline{}
	
	\centerline{}
	
	\title[On stronger forms of Devaney chaos]{On stronger forms of Devaney chaos}
	
	\author[Shital H. Joshi, Ekta Shah]{Shital H. Joshi$^1$ \MakeLowercase {and} Ekta Shah$^2$}
	
	
	\address{$^{1}$Shree M. P. Shah Arts and Science College, Surendranagar. India and Department of Mathematics, Faculty of Science, The Maharaja Sayajirao University of Baroda, Vadodara. India}
	\email{\textcolor[rgb]{0.00,0.00,0.84}{shjoshi11@gmail.com}}
	
	\address{$^{2}$Department of Mathematics, Faculty of Science, The Maharaja Sayajirao University of Baroda, Vadodara. India}
	\email{\textcolor[rgb]{0.00,0.00,0.84}{ekta19001@gmail.com}}
	
	
	\subjclass[2020]{Primary 37B65, 37B99.}
	
	\keywords{Sensitivity, Devaney chaos, $P-$chaos}
		
	\begin{abstract}
		\noindent	We define and study stronger forms of Devaney chaos and name it as $\mathscr{F}-$Devaney chaos, where $\mathscr{F}$ is a family of subsets of $\mathbb{N}$. Examples of maps which is $\mathscr{F}_t-$Devaney chaotic but not $\mathscr{F}_{cf}-$Devaney chaotic,  $\mathscr{F}_s-$Devaney chaotic but neither $\mathscr{F}_t-$Devaney chaotic nor $\mathscr{F}_{cf}-$Devaney chaotic are discussed. Further, we show that for the maps on infinite metric space without isolated points, $\mathscr{F}-$sensitivity is a redundant condition in the definition $\mathscr{F}-$Devaney chaos. Here $\mathscr{F}=\mathscr{F}_s, \: \mathscr{F}_t, \: \mathscr{F}_{ts}$ or $\mathscr{F}_{cf}$. We also obtain conditions under which Devaney chaos implies $\mathscr{F}_s-$Devaney chaos or $\mathscr{F}_t-$Devaney chaos. Next, we define the concept of $\left(\mathscr{F}, \mathscr{G}\right)-P-$chaos and obtain conditions under which $\left(\mathscr{F}_1, \mathscr{G}_1\right)-P-$chaos implies $\mathscr{F}-$Devaney chaos for different families $\mathscr{F}_1$ and $\mathscr{G}_1$.
	
	\end{abstract}
	\maketitle
	
	\section{Introduction}

\smallskip
\noindent Theory of chaotic dynamical systems is one of the challenging fields of dynamical systems because of its existence  in several disciplines ranging from physics and chemistry to ecology and economics to computer science. 
Conventionally, the term “chaos” means “a state of confusion with no order”. However in the theory of chaotic dynamical systems, the term is defined in more precise form. In 1975, Li and Yorke \cite{ly1} are the first one to connect the term “chaos” with a map. 

\medskip
\noindent In the existing literature, several definitions of chaos have been proposed and studied in detail, though, there is still no definitive, universally accepted mathematical definition of chaos. In fact, as per the authors of \cite{ly}, it is impossible to accept one such universal definition of chaos.  Most acknowledged definition of chaos includes Li – Yorke chaos, distributional chaos, positive topological entropy and its variants, Devaney Chaos, stronger forms of transitivity such as weakly mixing, mixing. 

\medskip
\noindent One of the key ingredient in the definition of Devaney chaos is sensitivity. The term sensitive dependence on initial condition was first used by Ruelle in \cite{RD} and the concept of sensitivity from the work of Ruelle was introduced to topological dynamics by Auslander and Yorke in \cite{AY}. 

\medskip
\noindent By a dynamical system we mean a pair $(X,f)$, where $X$ is an infinite metric space without isolated points and $f: X\longrightarrow X$ is a continuous map. Roughly, a map $f$ is sensitive if given any region in the space $X$, there exists two points in the region and $n \in \N$ such that the $n^{th}$ iterate of the two points under the map $f$ are separated significantly. The largeness of the set of all $n \in \N$ where this \textquoteleft significant separation' or sensitivity occurs can be considered as a measure of how sensitive system is. On the basis of largeness of these sets various stronger forms of sensitivity, known as $\F-$sensitivity, was first studied by Moothathu \cite{TKS}, where $\F$ is a family of subsets of $\N$. In this paper we define and study stronger forms of Devaney chaos using $\F-$sensitivity.

\medskip
\noindent Another important property in the theory of dynamical systems is the shadowing property as it is close to the stability of system and to the chaoticity of the system. Arai and Chinen defined the concept of $P-$chaos using the shadowing property and studied its relationship with different types of chaos \cite{Arai}. In the same paper they have proved that $P-$chaotic maps on a continuum are Devaney chaotic.

\medskip
\noindent Several generalizations and variations of shadowing property are studied. One such variation allow larger jumps in the definition of pseudo orbits and also demand adjustments in the tracing of the pseudo orbits by allowing mistakes. Oprocha, in \cite{Oprocha}, defined this variation in the most general form by introducing the concept of $(\mathscr{F}, \mathscr{G})-$shadowing property. Here $\mathscr{F}, \mathscr{G}$ are families of subsets of $\N_0=\N\cup \{0\}$. We replace the shadowing property by $\left(\mathscr{F}, \mathscr{G}\right)-$shadowing property in the definition of $P-$chaos and study $\left(\mathscr{F}, \mathscr{G}\right)-P-$chaos and study its relationship with $\F-$Devaney chaos.

\medskip
\noindent The paper is organized in the following manner. In Section 2 we discuss preliminaries required for the content of the paper. In \cite{TKS}, Moothathu defined and studied the stronger forms of sensitivity. For a given family $\F$ of $\N$, we call it as $\F-$sensitivity. We show that a map $f$ is $\F-$sensitive if and only if the iterate of $f$, $f^n$ is $\F-$sensitive. Further, certain results regarding shift spaces are also obtained in section 2.  Using the definition of $\F-$sensitivity, we define stronger forms of Devaney's chaos and name it as 
$\F-$Devaney chaos in section 3. Various examples of $\F-$Devaney chaotic maps are discussed. Also, we give examples of map which is $\F_t-$Devaney chaotic but not $\F_{cf}-$Devaney chaotic, $\F_s-$Devaney chaotic but is neither $\F_t-$Devaney chaotic nor $\F_{cf}-$Devaney chaotic. Banks \textit{et al} proved that every transitive map with dense set of periodic points on an infinite metric space without isolated points is Devaney chaotic \cite{Bank}. We prove similar result for  $\F-$Devaney chaotic maps. Note that every $\F-$Devaney chaotic map is Devaney chaotic but converse need not be true in general.  We obtain several conditions under which Devaney chaotic maps are $\F-$Devaney chaotic. In Section 4, we define the notion of $(\F, \mathscr{G})-P-$chaos by replacing shadowing property with $(\F_1,\mathscr{G}_1)-$shadowing and obtain many conditions which implies $\F-$Devaney chaos for various families $\F_1, \mathscr{G}_1$ on $\N_0$.

\section{Preliminaries}
\subsection{Subsets of $\N$}

\noindent Let $A \subset \mathbb{N}$. Then $A$ is said to be a
\emph{co--finite} set if $\mathbb{N}\setminus A$ is finite,
\emph{thick} set if $A$ contains arbitrary large blocks of consecutive numbers,
\emph{syndetic} set if $\mathbb{N}\setminus A$ is not thick (equivalently, it has a bounded gaps),
\emph{thickly  syndetic} set if for every $n\in \mathbb{N}$ the positions where length $n$ block begin forms a syndetic set, and
\emph{piecewise syndetic} set if it is an intersection of thick and syndetic sets. 
\noindent A \emph{family} on $\mathbb{N}$ is any set $\mathscr{F}\subset \mathscr{P}(\mathbb{N})$, which is upward hereditary; that is; if $A\in \mathscr{F}$ and $A\subset B\subset \mathbb{N}$ then $B\in \mathscr{F}$. The family of syndetic subsets of $\N$ is denoted by $\mathscr{F}_s$, the family of thick subsets of $\N$ is denoted by $\mathscr{F}_t$, the family of thickly syndetic subsets of $\N$ is denoted by $\mathscr{F}_{ts}$, the family of piecewise syndetic subsets of $\N$ is denoted by $\mathscr{F}_{ps}$ and the family of co--finite subsets of $\N$ is denoted by $\mathscr{F}_{cf}$.

\noindent The \emph{lower density} of $A\subset \N$ is defined by $$\underline{d}(A)=\liminf_{n\to \infty}\frac{|A\cap\{0,1,\dots,n-1\}|}{n}.$$ If we replace $\liminf$ with $\limsup$ in the above formula, we get $\overline{d}(A)$, the \emph{upper density} of $A$. If $\underline{d}(A)=\overline{d}(A)$, then the common value is said to be \emph{density} of $A$. The family of subsets of $\N$ with positive lower density is denoted by $\F_{\underline{d}}$. The family of subsets of $\N$ with density $1$ is denoted by $\mathscr{D}$.

\subsection{Dynamical Systems}

\noindent Recall by a \emph{dynamical system} we mean a pair $(X,f)$, where $X$ is an infinite metric space without isolated points and $f:X\longrightarrow X$ is a continuous map.  If $X$ is a compact metric space then we call $(X,f)$ as a \emph{compact dynamical system}. The \emph{orbit} of a point $x \in X$ is denoted by $O_f(x)$ and is defined by $O_f(x)=\{f^n(x):n \in \mathbb{N}_0\}.$  A point $x\in X$ is called a  \emph{periodic  point} of $f$ with prime period $m$ if $m$  is the smallest positive integer such that $f^ m(x)=x$. The set of all periodic points of the system is denoted by $P(f)$. For $x\in X$ and $U, V\subset X$, let $N_f(x,U)=\{n\in \mathbb{N}:f^n(x)\in U\}$ and $N_f(U,V)=\{n\in \mathbb{N}:f^n(U)\cap V\neq \phi\}$. A map $f$ is called \emph{non-wandering} if $N_f(U,U)\neq \phi$ for every nonempty open subset $U$ of $X$. 

\smallskip
\noindent A subset $Y$ of $X$ is called \emph{invariant} if $f(Y)\subseteq Y$. Map $f$ is called \emph{minimal} if for each point $x\in X$ $O_f(x)$ is dense in $X$. An invariant subset $Y\subset X$ is called \emph{minimal set} if the system $(Y,f|_Y)$ is a minimal dynamical system. A point $x\in X$ is called a \emph{minimal point} if it lies in some minimal subset of $X$. The set of all minimal points of $f$ is denoted by $M(f)$. Clearly, $P(f)\subset M(f)$. For a non--empty open subset $U$ of $X$ and a positive real number $\delta$, denote $$ {N_f (U,\delta)=\{n \in \mathbb{N}:\mbox{there exist}\ x, y \in U \  \mbox{with} \  d(f^{n}(x),f^{n}(y))>\delta\}.} $$ 

\smallskip	
\noindent The notion of sensitivity was first studied in \cite{AY}. We recall the definition.
\begin{definition}
	A map $f$ is said to be \emph{sensitive} if there exists $\delta>0$ such that for every nonempty open subset $U$ of $X$, $N_f (U,\delta)\neq \phi$. The constant $\delta$ is called \emph{sensitivity constant} for $f$.
\end{definition}

\smallskip
\noindent		The notion of stronger forms of sensitivity were first studied in \cite{TKS}. 
\begin{definition}
	Let $\mathscr{F}$ be a family on $\N$. A map $f$ is said to be \emph{$\mathscr{F}-$sensitive} if there exists $\delta > 0$ such that for every nonempty open set $U \subset X$, $N_f (U,\delta)\in \mathscr{F}$. 
\end{definition}

\medskip
\noindent Since we could not find the proofs of certain properties of $\F-$sensitive maps in the literature, we discuss the proof these properties here. In the following Theorem we obtain necessary and sufficient condition for iterates of $\F-$sensitive map to be $\F-$sensitive. 
\medskip

\begin{lemma} \label{T6}
	Let $(X,f)$ be a dynamical system. If $f$ is uniformly continuous and $\F-$sensitive then for any $m\in \N$, $f^m$ is $\F-$sensitive, where $\F=\F_s$, $\F_t$, $\F_{ts}$ or $\F_{cf}$.
\end{lemma}
\begin{proof}
	\noindent Suppose a uniformly continuous map $f$ is $\F-$sensitive with sensitivity constant $\delta$. We show that for $m\in \N$, $\phi =f^m$ is $\F-$sensitive.  By uniform continuity of $f^i$, $i=1,2,\cdots,m$, there exists $\eta>0$ such that for $x,y\in X$
	\begin{equation} \label{eq1}
		d(f^i(x),f^i(y))\geq \delta \implies d(x,y)> \eta.
	\end{equation}
	Let $U$ be a non--empty open subset of $X$. Then $\F-$sensitivity of $f$ implies that the set $N_f(U,\delta)\in \F$. Now, for each $n\in N_f(U,\delta)$, there exists $q_n, r_n \in \mathbb{N}_0$ such that $n=q_nm+r_n$, where 
	$0< r_n \leq m$. Therefore by Equation \ref{eq1}, for each $n \in N_f(U, \delta)$, we obtain
	$$ d\left (\phi^{q_n}(x),\phi^{q_n}(y) \right )>\eta.$$
	Set $B=\left\{q_n \in \N  \; | \; 0 < n-q_n m \leq m, \; n\in N_f(U,\delta)\right\}$. Then $B \subset N_{\phi}(U, \eta)$. It is sufficient to show that $B\in \F$ as $B\in \F$ implies $N_{\phi}(U,\eta)\in \F$. We consider the following different cases. 
	
	\smallskip					
	\noindent \textbf{Case 1:} $\F=\F_s$. 
	
	\noindent Without loss of generality we can assume that elements of $N_f(U, \delta)$ is in increasing order. Let $M$ be the bounded gap between the consecutive elements of $N_f(U, \delta)$. Then, for $q_n \in B$, it follows that 				
	$$q_n m < n \mbox{ and } q_n m \geq n-m.$$
	This further implies that 
	$$q_n-q_{n-1}<\frac{\displaystyle{n}}{\displaystyle{m}}+1-\frac{\displaystyle{n-1}}{\displaystyle{m}}<\frac{\displaystyle{M}}{\displaystyle{m}}+1$$
	Put $M'=\left[\frac{\displaystyle{M}}{\displaystyle{m}}+1\right]+1$. Then for $n\in N_f(U, \delta)$, 				
	$$q_n-q_{n-1} < M'.$$ 
	Hence, $B\in \F_s$. 
	
	\smallskip			
	\noindent \textbf{Case 2:} $\F=\F_t$. 
	
	\noindent Then by using the techniques similar as in Theorem 3.14 of \cite{ES}, it is easy to prove that $B\in \F_t$.
	
	\smallskip
	\noindent \textbf{Case 3:} $\F=\F_{ts}$
	
	\noindent Note that $B=\{q_n\in \mathbb{N}:\frac{n}{m}-1\leq q_n < \frac{n}{m},\; n\in N_f(U,\delta)\}.$ Since $N_f(U,\delta)\in \F_{ts}$, it follows that $N_f(U,\delta)\in \F_s$ with bounded gap $M$ and also $N_f(U,\delta)\in \F_t$. Therefore, by Case 2, $B\in \F_t$. We complete the proof by showing that for any $N\in \N$, $$B_N=\{r\in \mathbb{N}:r,r+1,\dots, r+N\in B\}\in \F_s.$$ Since $B\in \F_t$ it follows that for any $N\in \mathbb{N}$, there are infinitely many blocks of length $N$ in $B$ and therefore $B_N$ is infinite. Let $r_i,r_{i+1} \in B_N$ with $r_i<r_{i+1}$. Then there exist $n_i$ and $n_{i+1} \in N_f(U,\delta)$ such that $$\frac{n_i}{m}-1 \leq r_i < \frac{n_i}{m} \mbox{ and } \frac{n_{i+1}}{m}-1 \leq r_{i+1} < \frac{n_{i+1}}{m}.$$ Therefore, $$r_{i+1}-r_i<1+\frac{M}{m}.$$ Since $i$ is arbitrary it follows that $B_N$ has bounded gaps. Hence, $B_N\in \F_{s}$.

	\smallskip
	\noindent \textbf{Case 4:} $\F=\F_{cf}$
	
	\noindent Clearly, for given $n\in N_f(U,\delta)$, there is $q_n\in \N$ such that $\frac{\displaystyle{n}}{\displaystyle{m}}-1 \leq q_n < \frac{\displaystyle{n}}{\displaystyle{m}}$. Therefore $q_n\in B$. Further, $\frac{\displaystyle{n}}{\displaystyle{m}} \leq q_n+1 < \frac{\displaystyle{n}}{\displaystyle{m}}+1$. Hence $n+m\in N_f(U,\delta)$, implies $q_n+1\in B$. In general, for given $i\in \mathbb{N}$, $n+im\in N_f(U,\delta)$ implies $q_n+i\in B$. Since $N_f(U,\delta)\in \mathscr{F}_{cf}$, it follows that $q_n+i\in B$ for all $i\in \N$. Hence, $B\in \F_{cf}$.
\end{proof}

\begin{lemma}\label{fnf}
	Let $(X,f)$ be a dynamical system such that $f$ is uniformly continuous. Suppose for some $n\in \N$, $f^n$ is $\F-$sensitive. Then $f$ is $\F-$sensitive, where $\F=\F_s$, $\F_t$, $\F_{ts}$ or $\F_{cf}$.
\end{lemma}
\begin{proof}
	Let $\delta$ be a sensitivity constant for $f^n$ and let $U$ be a non--empty open subset of $X$. Then $N_{f^n}(U,\delta)\in \F$. By uniform continuity of $f^i$, $0\leq i\leq n-1$, there exists $\eta$, $0<\eta<\delta$, such that $d(f^i(x),f^i(y))\geq \delta\implies d(x,y)>\eta$. In each of the following cases we show that $N_f(U,\eta)\in \F$.
	
	\medskip			
	\noindent Case 1. $\F=\F_s$
	
	\noindent Note that 
	$$N_{f^n}(U,\delta) = \left\{k\in\N \; : \; \exists \; y,z \in U \ni d\left((f^n)^k(y),(f^n)^k(z)\right) > \delta \right\}$$
	But this implies $n N_{f^n}(U,\delta)\subset N_f(U,\delta)$. Also, $N_{f^n}(U,\delta)\in \FS$ implies $nN_{f^n}(U,\delta)\in \F_s$ \cite[Theorem 3.1]{KP}. Therefore, $N_f(U,\delta)\in \FS$.
	
	\medskip			
	\noindent Case 2. $\F=\F_t$
	
	\noindent Then by arguments as used in Theorem 3.14 of \cite{ES}, it follows that $N_f(U,\eta)\in \F_t$.
	
	\medskip
	\noindent Case 3. $\F=\F_{ts}$
	
	\noindent Denote $N_{f^n}(U,\delta)=A$. Then $A\in \F_{ts}$. Set $B=\bigcup_{q=0}^{n-1}(nA-q).$ It is sufficient to show that $B\in \F_{ts}$ as $B\subset N_f(U,\eta)$. Since $A\in \F_{ts}$ it follows that $A\in \F_t$ and hence by Case 2, $B\in \F_t$. For each $N\in \N$, $B_N=\{r\in \N: r,r+1,\dots, r+N\in B\}$. Clearly, $B_N$ is infinite. Let $r_i, \; r_{i+1}\in B_N$ with $r_i<r_{i+1}$. Then $r_i=nn_1-q_1$ and $r_{i+1}=nn_2-q_2$ for some $n_1, \; n_2 \in A$ and $0\leq q_1, \; q_2 \leq n-1$. It is easy to observe that $0<r_{i+1}-r_i \leq n(M+1)-1$, where $M$ is a bounded gap for $A$. Therefore $B_N\in \F_{s}$.

	\medskip
	\noindent Case 4. $\F=\F_{cf}$
	
	\noindent Denote $N_{f^n}(U,\delta)=A$. Then $A\in \F_{cf}$. Set $B=\bigcup_{q=0}^{n-1}(nA-q).$ Clearly, $B\in \mathscr{F}_{cf}$ and $B\subset N_f(U,\eta)$. Therefore $N_f(U,\eta)\in  \mathscr{F}_{cf}$
	
	\smallskip
	\noindent Thus in any case $N_f(U,\eta)\in \F$. Hence $f$ is $\F-$sensitive with sensitive constant $\eta$.	
\end{proof}

\begin{proposition}
	Let $(X,f)$ and $(Y,g)$ be compact conjugate dynamical systems. Then $f$ is $\F-$sensitive if and only if $g$ is $\F-$sensitive, where $\F$ is a family of subsets of $\N$. That is, $\F-$sensitivity is preserved under conjugacy.
\end{proposition}
\begin{proof}
	Let $f$ be $\F-$sensitive with sensitivity constant $\delta$. It is sufficient to take open balls instead of open sets because any open set contains open ball. Let $V=B_{\epsilon}(y)$ for some $y\in Y$ and $\epsilon>0$. Since $f$ is onto map, it follows that there exists $x\in X$ such that $h(x)=y$ and $h^{-1}(y)=x$. By uniform continuity of $h$ there exists $\theta>0$ such that $d(p,q)<\theta \implies d(h(p),h(q))<\epsilon$. Let $U=B_{\theta}(x)$. Then $N_f(U,\theta)\in \F$. Let $n\in N_f(U,\theta)$. Then there exists $x_1,x_2\in U$ such that 
	\begin{equation}\label{new}
		d(f^n(x_1),f^n(x_2))>\delta. 
	\end{equation}
	Let $h(x_1)=y_1$ and $h(x_2)=y_2$ then we have $h^{-1}(y_1)=x_1$ and $h^{-1}(y_2)=x_2$. Clearly, $y_1,y_2\in V$. By uniform continuity of $h^{-1}$ there exists $\eta>0$ such that $\rho(h^{-1}(s),h^{-1}(t))\geq \delta \implies \rho(s,t)\geq \eta$. By inequality \ref{new} and conjugacy, it is easy to prove that 
	$\rho(g^n(y_1),g^n(y_2))>\eta$. That is $n\in N_g(V,\eta)$. This implies that $N_f(U,\delta)\subset N_g(V,\eta)$. Therefore $N_g(V,\eta)\in \F$. Hence $g$ is $\F-$sensitive with sensitivity contant $\eta$.  
\end{proof}

\noindent Let $(X, f)$ be a dynamical systems. Then map $f$  is said to be \emph{transitive} if for any two non--empty open $U$ and $V$ of $X$, $N_f(U,V)\neq \phi$. In fact, it can be observed that $N_f(U,V)$ is an infinite set. Therefore, it is possible to define the stronger forms of transitivity. We recall the definition.

\smallskip	
\begin{definition}
	Let $(X, f)$ be a dynamical system. For a family $\mathscr{F}$ on $\N$, map $f$ is called \emph{$\mathscr{F}-$transitive} if $N_f(U,V)\in \mathscr{F}$. 
\end{definition}

\smallskip
\noindent Map $f$ is said to be a \textit{mixing} if $N_f(U,V)$ is a co--finite set, for any non---empty open subsets $U, V$ of $X$. Note that $f$ is a mixing is equivalent to $f$ is $\F_{cf}-$transitive. Next, map $f$ is a \emph{weakly mixing map}  if $f\times f$ is transitive. Therefore, weak mixing of $f$ is equivalent to $f$ is $\F_t-$transitive. A map $f$ is \emph{totally transitive} if $f^n$ is transitive for all $n\in \mathbb{N}$.  The notion of Devaney chaos was first introduced in \cite{RL}. We recall the definition.

\smallskip
\begin{definition}
	Let $(X,f)$ be a dynamical system. 	A map $f$ is said to be \emph{chaotic in the sense of Devaney or Devaney chaotic} if it is transitive, sensitive and $P(f)$ is dense in $X$.
\end{definition}

\smallskip
\noindent Fix $\epsilon,\; \delta>0$. A sequence $\{x_i:i\geq 0\}$ is a \emph{$\delta-$pseudo orbit} for map $f$ if $d(f(x_i),x_{i+1})<\delta$, for all $i\geq 0$. A point $x\in X$ is said to $\epsilon-$shadow (or traces) a $\delta-$pseudo orbit $\{x_i:i\geq 0\}$ if $d(f^i(x),x_i)<\epsilon$. Points $x,y\in X$ are \emph{chain related} if for every 
$\delta >0$ there are finite $\delta$-pseudo orbits from $x$ to $y$ and from $y$ to $x$. We say $x\in X$ is a \emph{chain recurrent} point for $f$ if $x$ is chain related to $x$. The set of all chain recurrent points is denoted by $CR(f)$. The map $f$ is \emph{chain recurrent} if $CR(f)=X$. The map $f$ is called \emph{chain transitive} if any two points of $X$ are chain related.

\subsection{Shift Space}
\noindent Let $\Sigma=\{x=(x_0x_1\dots):x_i =0 \text{ or } 1,\; i\geq 0\}$. Then $\Sigma$ is a metric space with the metric 

$$d\left(x, y\right)=\left \{\begin{array}{lll} 0, & \textit{ if } x=y \\ 
	2^{-J(x,y)}, & \textit{if } x\neq y \end{array} \right.$$ 
where $J(x,y)=\textit{min}\{i\in \N_0:x_i\neq y_i\}$. A \emph{shift map} is a function $\sigma : \Sigma \longrightarrow \Sigma$ given by $\sigma(x_0x_1x_2\dots)=(x_1x_2x_3\dots)$. A \emph{subshift} is the restriction of $\sigma$ to any closed non--empty subset $X$ of $\Sigma$ that is invariant under $\sigma$. A $k-$word in $x=(x_0x_1x_2\dots)$ is a finite sequence of the form $x_ix_{i+1}\dots x_{i+k-1}$ for some $i\in \N_0$. The language $L(X)$ of a subshift $X$ is the set of all words that appear in some element of $X$ together with empty word $\epsilon$. For an $(m+1)-$word $u=u_0u_1\dots u_m$, the set $C[u]=\{x\in X:x_i=u_i \mbox{ for } 0\leq i\leq m\}$ is called the \emph{cylinder set}. In the following we give the characterization of subshift to be $\F_s-$transitive.

\begin{proposition}\label{p1}
	Let $(X,\sigma)$ be a subshift space. Then $\sigma$ is $\F_s-$transitive on $X$ if and only if the set $\{|w|:uwv\in L(X)\}\in \F_s$ for all $u,v\in L(X)$.
\end{proposition}

\begin{proof}
	Suppose $\sigma$ is $\F_s-$transitive on $X$. Let $u=u_0u_1\dots u_m$ and $v=v_0v_1 \dots v_k \in L(X)$. Consider the cylinder sets $C[u]=\{x\in X:x_i=u_i \text{ for } 0\leq i\leq m\}$ and $C[v]=\{x\in X:x_i=v_i \text{ for } 0\leq i\leq k\}$. Then $N_{\sigma}(C[u],C[v])\in \F_s$. For any $n>m$ and $n\in N_{\sigma}(C[u],C[v])$ there exists $w\in L(X)$ such that $|w|=n-m$ and $uwv\in L(X)$. Let $A=N_{\sigma}(C[u],C[v])\setminus \{1,2,\dots,m\}$. Then $A\in \F_s$ and therefore $\{|w|:uwv\in L(X)\}=A-m=\{a-m:a\in A\} \in \F_s$.
	
	\medskip
	\noindent Conversely, suppose that $\{|w|:uwv\in L(X)\}\in \F_s$ for any $u,v\in L(X)$. Let $U$ and $V$ be any non--empty subsets of $X$. Then there exists $u=u_0u_1\dots u_{m-1}$ and $v=v_0v_1\dots v_k \in L(X)$ such that $C[u]\subset U$ and $C[v]\subset V$. It is sufficient to show that $N_{\sigma}(C[u],C[v])\in \F_s$ because $N_{\sigma}(C[u],C[v])\subset N_{\sigma}(U,V)$. By assumption, the set $A=\{|w|:uwv\in L(X)\}\in \F_s$. For any $n\in A$, there exists $w\in L(X)$ such that $uwv\in L(X)$. Therefore, $m+n\in N_{\sigma}(C[u],C[v])$. Hence, $m+A\subset  N_{\sigma}(C[u],C[v])$. This implies that $N_{\sigma}(C[u],C[v])\in \F_s$. This completes the proof.
\end{proof}

\noindent By similar arguments used in the proof of Proposition \ref{p1} we can prove that $\sigma$ is $\F_{ts}-$transitive on $X$ if and only if the set $\{|w|:uwv\in L(X)\}\in \F_{ts}$, for all $u,v\in L(X)$.  

\medskip
\noindent The notion of spacing shift was defined and studied in \cite{Banks}. We recall the definition. For any $P\subset \N$, the set $\Sigma_P=\left\{s\in \Sigma:s_i=s_j=1\implies |i-j|\in P\cup \{0\}\right\}$ is a subshift of $\Sigma$. The subshift defined in this way is called a \emph{spacing shift}. Characterization of $\sigma$ to be mixing, weak mixing on $\Sigma_P$ in terms of $P$ was obtained in \cite{Banks}.

\smallskip
\begin{lemma}\cite{Banks} \label{T00}
	For a subset $P$ of $\N$, consider the spacing shift $\Sigma_P$. Then the following holds:\\
	1. $\Sigma_P$ mixing $\iff$ $P$ is co--finite.\\
	2. $\Sigma_P$ weakly mixing $\iff$ $P$ is thick.\\
	3. For $P\neq \phi$, $\Sigma_P$ has a dense set of periodic points if and only if for every $p\in P$ there exists $k\in \N$ such that $k\N \cup (k\N+p) \cup (k\N-p)\subseteq P$. 
\end{lemma}

\noindent In the following example we give characterization on a subset $P$ of $\N$ for $\sigma$ on $\Sigma_P$ to be $\F_s-$transitive.

\begin{example}\label{eg1} Let $P$ be a subset of $\N$ and $\Sigma_P$ be the corresponding spacing shift. Then $\sigma$ on $\Sigma_P$ is $\F_s-$transitive if $P\in \F_{ts}$. Conversely, if $\sigma$ on $\Sigma_P$ is $\F_s-$transitive then $P\in \F_s$.
\end{example}

\begin{proof}
	Let $u,v\in L(\Sigma_P)$ such that $u=x_0x_1\dots x_{m-1}$ and $v=y_0y_1\dots y_{n-1}$. Since $P\in \F_{ts}$ it follows that $P$ has block of length $m+n$. Let $P_1$ be the set of positions where blocks of length $m+n$ begin in $P$. Then $P_1\in \F_s$. Put $z=u0^kv$, where $k\in P_1$. Let $i,j\in \N$ with $i\leq j$. Then the following three possibilities: 
	\begin{enumerate}
		\item If $j\leq m-1$,  then $z_i=x_i$ and $z_j=x_j$. Therefore $j-i\in P$. 
		\item If $i\geq m+k$ then $z_i=y_{i-(k+m)}$ and $z_j=y_{j-(k+m)}$. Therefore $j-i\in P$. 
		\item If $i\leq m-1$ and $j\geq k+m$ then $k\leq j-i\leq k+m+n-1$. Hence $j-i\in P$. Thus, $\{|w|:uwv\in L(\Sigma_P)\}=P_1\in \F_s$ by Proposition \ref{p1}. 
	\end{enumerate}

	\noindent Conversely, suppose $\sigma$ is $\F_s-$transitive then it is transitive and hence $P$ is infinite \cite{Banks}. In particular, for $u=v=1$, put $P'=\{|w|:1w1\in L(\Sigma_P)\}\in \F_s$. Then $P'=\{|w|:|w|+1\in P\}=P-1$. Since $P'\in \F_s$, it follows that $P=P'+1\in \F_s$. This completes the proof.
\end{proof}

\noindent  By similar arguments used in the Example \ref{eg1} we can prove that if $\sigma$ on $\Sigma_P$ is $\F_{ts}-$transitive then $P\in \F_{ts}$.

\medskip
\section{$\mathscr{F}-$Devaney's Chaos}

\medskip
\noindent Recall, a continuous map on a metric space $X$ is Devaney chaotic if $f$ is transitive,  sensitive and $P(f)$ is dense in $X$. In the following we define stronger forms of Devaney's chaos using stronger forms of sensitivity and transitivity. 

\begin{definition}
	Let $(X, f)$ be a dynamical system and let $\F$ be a family on $\N_0$. Then map $f$ is said to be \emph{$\F-$Devaney chaotic} if $f$ is $\F-$transitive, $\F-$sensitive and $P(f)$ is dense in $X$.
\end{definition}

\smallskip
\noindent If $\F=\FS$, then $f$ is said to be \emph{syndetically Devaney chaotic}, if $\F=\F_t$, then $f$ is said to be \emph{thickly Devaney chaotic}, if $\F=\F_{ts}$, then $f$ is said to be \emph{thickly syndetic Devaney chaotic}, if $\F=\F_{cf}$, then $f$ is said to be \emph{co--finite Devaney chaotic}.

\medskip
\noindent  In \cite{Bank}, Banks \textit{et al}, proved that if $f$ is a transitive map defined on an infinite metric space $X$ without isolated points and $P(f)$ is dense in $X$, then $f$ is sensitive and hence Devaney chaotic. We prove similar result for $\F-$Devaney chaos.

\begin{theorem} \label{T35}
	Let  $X$ be an infinite metric space without isolated points and $f: X \longrightarrow X$ be a uniformly continuous map. Suppose $f$ is $\F-$transitive and $P(f)$ is dense in $X$.  Then $f$ is $\F-$sensitive and hence $\F-$Devaney chaotic, where $\F=\F_{s}$, $\F_t$, $\F_{ts}$ or $\F_{cf}$. 
\end{theorem}	

\begin{proof}
	In view of Lemma \ref{fnf}, it is sufficient to prove that $f^p$ is $\F-$sensitive for some $p \in \N$. Since $P(f)$ is dense in $X$ and $X$ is infinite, it follows that $P(f)$ is infinite. Let $u$ and $v$ be two periodic points such that $O_f(u)\cap O_f(v)=\phi$ and $8\delta =min\left\{d\left(f^m(u), f^n(v)\right) \; : \; m, n \geq 0 \right\}$. Then $\delta > 0$ and for each $t\in X$, $d\left(t, f^k(q)\right)\geq\delta/2$, for all $k\geq 0$. Here either $q=u$ or $q=v$.
	
	\medskip
	\noindent	Let $x \in X$ and $W$ be a neighbourhood of $x$. Put $U=W\cap B(x,\delta)$. Since $P(f)$ is dense in $X$, it follows that there is a periodic point $y$, say of period $n$, such that $y \in U$. Set $\eta = \delta/8$. Then we show that $N_{f^n}(W,\eta)=\{k\in \N: \mbox{there exists $u$, $v\in W$ with } d\left((f^n)^k(u),(f^n)^k(v)\right)>\eta\}$ contains an element of $\F$. 
	
	\smallskip
	\noindent For $x\in X$, let $q$ be a periodic point such that $d(x,f^k(q))\geq \frac{\delta}{2}$, for all $k\geq 0$. Set $$V=\bigcap_{i=0}^{n}f^{-i}\left(B(f^i(q),\eta)\right).$$ Then $V$ is a non--empty open subset of $X$ by the choice of $\eta$. 
	
	\medskip
	\noindent Since $f$ is $\mathscr{F}-$transitive, for this $U,V$, it follows that $N_f(U,V)\in \mathscr{F}$. Therefore for each $m\in N_f(U,V)$ there exists $t\in U$ such that $f^m(t)\in V$. Put 
	$$A=\left\{j\in \N:1\leq nj-m\leq n, \; m\in N_f(U,V)\right\}.$$ 
	Then $A$ is infinite as $N_f(U,V)$ is infinite. Let $j\in A$. Then $1\leq nj-m\leq n$. Further,
	$$f^{nj}(t)\in f^{nj-m}(V)\subseteq B(f^{nj-m}(q),\eta).$$
	Also, $f^{nj}(y)=y$ implies $ d(f^{nj}(t),f^{nj}(y))>2\eta $ as $y\in B(x,\eta)$ and $f^{nj}(t)\in B(f^{nj-m}(q),\eta)$. This further implies $d(f^{nj}(y),f^{nj}(x))>\eta$ or $d(f^{nj}(x),f^{nj}(t))>\eta$. Therefore $j\in N_{f^n}(W,\eta)$ and hence $A\subset N_{f^n}(W,\eta)$. We complete the proof by showing that $A\in \F$. Consider the following five cases.
	
	\medskip
	\noindent \textbf{Case 1:} $\F=\FS$
	
	\noindent Here $N_f(U,V)\in \F_s$. With the similar arguments in Case 1 of Lemma \ref{T6} one can show that $A\in \F_s$.  
	
	\medskip
	\noindent \textbf{Case 2:} $\F=\F_t$
	
	\noindent Let $k\in \N$. Since $N_f(U,V)\in \F_t$, it follows that there is a block of length $kn$ in $N_f(U,V)$. Therefore, there exists $a\in N_f(U,V)$ satisfying $$\{a,a+1,\cdots, a+n, \cdots, a+2n,\cdots , a+kn\}\subset N_f(U,V).$$
	Further, there exists $r\in \mathbb{N}$ such that $\frac{\displaystyle{a}}{\displaystyle{n}}< r \leq \frac{\displaystyle{a}}{\displaystyle{n}}+1.$ But this implies $r\in A$. Next, $\frac{\displaystyle{a+n}}{\displaystyle{n}}=\frac{\displaystyle{a}}{\displaystyle{n}}+1$ and $\frac{\displaystyle{a+n}}{\displaystyle{n}}+1=\frac{\displaystyle{a}}{\displaystyle{n}}+2$ implies
	$\frac{\displaystyle{a}}{\displaystyle{n}}+1 < r+1 \leq \frac{\displaystyle{a}}{\displaystyle{n}}+2$, which further implies $r+1\in A$. For $0\leq i<k$, $r+i\in A$. Hence $\{r,r+1,\dots,r+k-1\}\subset A$. But $k\in \N$ is arbitrary. Thus, $A\in \F_t$.
	
	\medskip
	\noindent \textbf{Case 3:} $\F=\F_{ts}$
	
	\noindent Here $N_f(U,V)\in \F_{ts}$. With the similar arguments in Case 3 of Lemma \ref{T6} one can show that $A\in \F_{ts}$.

	\medskip
	\noindent \textbf{Case 4:} $\F=\F_{cf}$ 
	
	\noindent By Case 2, for given $a+in\in N_f(U,V)$, $r+i\in B$. Since $N_f(U,V)\in \F_{cf}$ it follows that $A\in \F_{cf}$.
	
	\smallskip
	\noindent Thus in any case $A\in \F$. Therefore $N_{f^n}(U,\eta)\in \F$. Hence, $f$ is $\F-$sensitive and therefore $\F-$Devaney chaotic.
\end{proof}

\medskip
\noindent In the following  we give an Example of a map which is $\F_s-$transitive, $\F_s-$sensitive but the set of periodic points is not dense and therefore map is not $\F_s-$Devaney chaotic.

\smallskip
\begin{example}
	Consider the shift space $(\Sigma,\sigma)$. For a real number $r$ denote the fractional part of $r$ by \emph{frac$(r)$}. Let $\alpha\in (0,1)$ be an irrational number. Define an element $x^{\alpha}\in \Sigma$ as follows: 
	\[ x_n^{\alpha} = \left\{ \begin{array}{lll}
		0, & \mbox{if  $frac(n\alpha)\in[0,1-\alpha)$}\\
		1,  & \mbox{if  $frac(n\alpha)\in[1-\alpha,1)$}.\\
	\end{array}
	\right. \]
	\noindent Then the subset $X_{\alpha}=\overline{O_{\sigma}(x^{\alpha})}$ of $\Sigma$ together with the shift map $\sigma$ is called a \emph{Sturmian subshift}. Since $\sigma$ is minimal on  $X_{\alpha}$, it follows that it is $\F_s-$transitive and it does not have dense set of periodic points. Also, $\sigma$ is $\F_s-$sensitive \cite{TKS}. Therefore, $\sigma$ is not $\F_s-$Devaney chaotic.
\end{example}

\smallskip
\noindent In the following example we give characterization of a subset $P$ of $\N$ for the shift map $\sigma$ on $\Sigma_P$ to be $\F_{cf}-$Devaney chaotic.

\smallskip
\begin{example}
	Let $P$ be a co--finite subset of $\N$. Then the spacing shift $\Sigma_P$ is a mixing by Lemma \ref{T00} (1). Since $P$ is a co--finite set, for any $p\in P$ there exists $k\in \N$ such that $k\N\cup (k\N+p) \cup (k\N-p)\subseteq P$. Therefore by Lemma \ref{T00} (3), the set of periodic points of $\sigma$ is dense in $\Sigma_P$. Hence by Theorem \ref{T35}, $\sigma$ on $\Sigma_P$ is $\mathscr{F}_{cf}-$sensitive and hence $\sigma$ is $\mathscr{F}_{cf}-$Devaney chaotic. In view of Lemma \ref{T00} (2), converse is also true.
\end{example}

\smallskip
\noindent In the following we give example of a map which is $\F_t-$Devaney chaotic but not $\F_{cf}-$Devaney chaotic.
\begin{example}
	Let $P=\N\setminus \{2^n:n\in \N\}$. Then $P$ is a thick subset of $\N$. For any $p\in P$ putting $k=2p$ gives $k\N\cup (k\N+p) \cup (k\N-p)\subseteq P$. Since $P$ is thick, it follows that $\Sigma_P$ is weakly mixing. Also, periodic points are dense in $\Sigma_P$ by Lemma \ref{T00} $(3)$. Hence shift map is $\F_t-$Devaney chaotic. Note that $P$ is not co--finite and hence shift map is not $\F_{cf}-$Devaney chaotic on $\Sigma_P$.
\end{example}

\medskip
\noindent Let $(X,f)$ be a compact dynamical system. Then it is known that map $f$ is weakly mixing if and only if for any two non--empty open sets $U,V\subset X$, $N_f(U,V)$ is thick \cite{Huang}. Therefore weakly mixing is equivalent to $\F_t-$transitive. Further, if $f$ is weak mixing then it is $\mathscr{F}_t-$sensitive \cite{Liu}. Hence we have the following obvious consequence. 

\begin{proposition}
	Let $(X,f)$ be a compact dynamical system. If $f$ is weak mixing and $P(f)$ is dense in $X$, then $f$ is $\F_t-$Devaney chaotic. 
\end{proposition}

\medskip
\noindent In \cite{Liu}, it is shown that a Devaney chaotic map $f$ is thickly periodically sensitive and hence it is $\F_t-$sensitive. Therefore the following Theorem gives conditions for Devaney chaotic map to be $\F_t-$Devaney chaotic.

\smallskip		
\begin{proposition}\label{p3.7}
	Let $(X,f)$ be a compact dynamical system. Suppose $f$ is weak mixing. Then $f$ is Devaney chaotic map if and only if $f$ is $\F_t-$Devaney chaotic.
\end{proposition}

\smallskip
\noindent Note that weak mixing is a necessary condition in Proposition \ref{p3.7} is justified by the following example.

\smallskip		
\begin{example}\label{3.10}
	Consider spacing shift $(\Sigma_P,\sigma)$ for $P=2\N$. Then by Lemma \ref{T00} it follows that $\sigma$ is transitive on $\Sigma_P$ and $P(f)$ is dense in $\Sigma_P$. Also, as observed above weak mixing is equivalent to $\F_t-$transitive. By Lemma \ref{T00} $\sigma$ on $\Sigma_P$ is weak mixing if and only if $P$ is thick. Here $P=2\N$ is not a thick set. Therefore $\sigma$ is not $\F_t-$transitive. Therefore $\sigma$ is Devaney chaotic but not $\F_t-$Devaney chaotic.
\end{example}

\smallskip
\noindent Recall, if $f$ is transitive map with a dense set of minimal points then it is $\F_s-$transitive. Further, if $f$ is $\F_s-$transitive but not minimal, then $f$ is $\F_s-$sensitive \cite{TKS}. Hence a transitive non--minimal map $f$ with a dense set of minimal points is $\F_s-$sensitive. Therefore we have the following consequence.

\begin{proposition}\label{3.9}
	Let $(X,f)$ be a compact dynamical system. Suppose $f$ is a non--minimal map. Then $f$ is Devaney chaotic if and only if $f$ is $\F_s-$Devaney chaotic. 
\end{proposition}

\begin{proof}
	Since every periodic point is a minimal point it follows that $M(f)$ is dense in $X$. Further, $f$ is transitive implies $f$ is $\F_s-$transitive. Therefore by Theorem \ref{T35}, $f$ is $\F_s-$Devaney chaotic. 
\end{proof}

\smallskip 
\noindent In the following we give an example of a map which $\F_s-$Devaney chaotic but is neither $\F_t-$Devaney chaotic nor $\F_{cf}-$Devaney chaotic. Note that this is also an example of a map which is Devaney chaotic but is neither $\F_t-$Devaney chaotic nor $\F_{cf}-$Devaney chaotic.

\smallskip		
\begin{example}\label{2.6} The map $S:[-1,1]\longrightarrow [-1,1]$, defined by 
	\begin{center}
		$S(x)=\left\{ \begin{array}{llc}
			2x+2 & \mbox{ if } x\in [-1,-\frac{1}{2}],\\ 
			-2x & \mbox{ if } x\in [-\frac{1}{2},0], \\
			-x & \mbox{ if } x\in [0,1]
		\end{array} \right.$
	\end{center}
	is transitive but not topological mixing. It is known that transitive map on interval have dense set of periodic points \cite{Michel}. So, the map defined above is Devaney chaotic and hence $\F_s-$Devaney chaotic but not $\F_t-$Devaney chaotic.
\end{example}

\smallskip
\begin{theorem}\label{3.10}
	Let $(X,f)$ be a compact dynamical system. Suppose $f$ is non--minimal weakly mixing map with $P(f)$ is dense in $X$. Then $f$ is $\F_{ts}-$Devaney chaotic.
\end{theorem}

\begin{proof}
	Since $P(f)\subset M(f)$, it follows that $M(f)$ is dense in $X$. Also, non--minimal transitive map is $\F_s-$transitive\cite{TKS}. Then $f$ is $\F_{ts}-$transitive because $f$ is non--minimal weakly mixing map for which $N_f(U,U)\in \F_s$ \cite[Theorem 4.7 (1)]{Wen}. Thus, $f$ is $\F_{ts}-$Devaney chaotic by Theorem \ref{T35}.
\end{proof}

\section{$(\mathscr{F},\mathscr{G})-P-$chaos}

\medskip
\noindent Throughout this section $(X,f)$ is a compact dynamical system and we consider families of subsets of $\N_0$. The notion of $P-$chaos was first studied and defined in \cite{Arai}. A map $f$ is said to be \emph{$P-$chaotic} if $P(f)$ is dense in $X$ and $f$ has shadowing property. It is proved that every $P-$chaotic map $f$ defined on continuum is mixing and further, $f$ is Devaney chaotic. Also, if $f$ is mixing then it is $\F_{cf}-$sensitive \cite{TKS}. Therefore we have the following consequence.

\begin{proposition}\label{3.1}
	Every $P-$chaotic map defined on a continuum is $\F_{cf}-$Devaney chaotic.
\end{proposition}

\smallskip
\noindent Fix $A\subset \mathbb{N}_0$ and $\delta$, $\epsilon$ $>0$. A sequence $\{x_i\}_{i=0}^{\infty}$ in $X$ is a \emph{$\delta-$pseudo orbit on $A$} if $A\subset \{i:d(f(x_i),x_{i+1})<\delta\}.$ If $A=\N_0$, then $\{x_i\}_{i=0}^{\infty}$ is a $\delta$-pseudo orbit in classical sense. A point $y\in X$ is said to \emph{$\epsilon-$shadows (or $\epsilon-$traces)} a sequence $\{y_i\}_{i=0}^{\infty}$ on $B$ if $B\subset \{i:d(f^i(y),y_i)<\epsilon\}$. If $A=\mathbb{N}_0$ then $\{y_i\}_{i=0}^{\infty}$ is $\epsilon-$shadowed (or $\epsilon-$traced) by a point $y$. The notion of $(\F,\G)-$shadowing property was defined in \cite{Oprocha}. We recall the definition.

\begin{definition}
	Consider families $\mathscr{F}$ and $\mathscr{G}$ of $\mathbb{N}_0$. A map $f$ is said to have \emph{$(\mathscr{F},\mathscr{G})-$shadowing property} if for every $\epsilon>0$ there is $\delta>0$ such that if $\{x_i\}_{i=0}^{\infty}$ is a $\delta-$pseudo orbit on a set $A$ in $\mathscr{F}$ then there is a point $x$ which $\epsilon-$shadows $\{x_i\}_{i=0}^{\infty}$ on a set $B$, for some $B\in \mathscr{G}$. 
\end{definition}

\smallskip
\noindent It was observed in \cite{Oprocha}, that in general there is no relation between $(\F,\G)-$shadowing and shadowing property. Note that if $\F=\N_0=\G$, then $(\F,\G)-$shadowing is usual shadowing of $f$. Replacing shadowing by $(\F,\G)-$shadowing in $P-$chaos we define $(\F,\G)-P-$chaos. 
\begin{definition} 
	Let $(X,f)$ be a dynamical system and let $\mathscr{F}$ and $\mathscr{G}$ be families of $\mathbb{N}_0$. A map $f$ is said to be \emph{$(\mathscr{F},\mathscr{G})-P-$chaotic} if it has $(\mathscr{F},\mathscr{G})-$shadowing and periodic points of $f$ are dense in $X$.
\end{definition}

\noindent In the following we give examples of a map which is $(\F,\G)-P-$chaotic but not $P-$chaotic.
\begin{example}\label{2.11}
	Let $f:[0,1]\longrightarrow [0,1]$ be the piece--wise linear map defined by $f(0)=0$, $f(\frac{1}{6})=\frac{1}{2}$, $f(\frac{1}{3})=0$, $f(\frac{2}{3})=1$, $f(\frac{5}{6})=\frac{1}{2}$, $f(1)=1$. Then $P(f)$ is dense in $[0,1]$ but $f$ does not have shadowing property. Hence $f$ is not $P-$chaotic map. Let $M(f)$ be the set of all minimal points of $f$. Then $P(f)\subset M(f)$. Therefore $f$ has $(\mathscr{P}(\N_0),\F_{ps})-$shadowing property \cite{Oprocha} and $f$ is $(\mathscr{P}(\N_0),\F_{ps})-P-$chaotic.
\end{example}

\begin{example}\label{eg2}
	Let $X=\{p,q\}$. Consider a discrete metric on $X$ and let $f:X\longrightarrow X$ be an identity map. Then $P(f)$ is dense in $X$ and $f$ has $(\mathscr{D},\F_{t})-$shadowing property \cite{Oprocha2}. Therefore $f$ is $(\mathscr{D},\F_{t})-P-$chaotic.
	
\end{example}

\smallskip
\noindent In the following results we obtain the relationship between $(\F,\G)-P-$chaos and $\F-$Devaney chaos.

\begin{theorem}\label{3.4}
	Let $(X,f)$ be a compact dynamical system. If $f$ is chain mixing, totally transitive contraction map having $(\N_0,\F_s)-$shadowing property, then $f$ is $(\N_0,\F_{\underline{d}})-P-$chaotic if and only if $f$ is $\F_{cf}-$Devaney chaotic.
\end{theorem}
\begin{proof}
	If $f$ is chain mixing contraction mapping with $\F_s-$shadowing property, then $f$ has $(\N_0,\F_{\underline{d}})-$shadowing property if and only if $f$ is mixing \cite{darabi}. Since $P(f)$ is dense in $X$ and $f$ is mixing, the result follows from Theorem \ref{T35}.
\end{proof}

\begin{theorem}\label{3.5}
	Let $(X,f)$ be a dynamical system. Suppose $f$ is an onto map. Then $f$ is $(\mathscr{D},\mathscr{D})-P-$chaotic if and only if $f$ is $P-$chaotic and $\F_{cf}-$Devaney chaotic.
\end{theorem}
\begin{proof}
	If $f$ is an onto map on $X$ then $(\mathscr{D},\mathscr{D})-$shadowing is equivalent to shadowing and topological mixing \cite{Fakhari}. This implies that $(\mathscr{D},\mathscr{D})-P-$chaotic map is $P-$chaotic and $\F_{cf}-$Devaney chaotic. 
	
	\noindent Conversely, suppose $f$ is $\F_{cf}-$Devaney chaotic and $P-$chaotic then it has shadowing property and it is mixing and hence it has $(\mathscr{D},\mathscr{D})-$shadowing. Thus, $f$ is $(\mathscr{D},\mathscr{D})-P-$chaotic.
\end{proof}

\begin{theorem}\label{3.6}
	Let $(X,f)$ be a dynamical system. Suppose $f$ is chain mixing. If $f$ is $(\N_0,\F_{\underline{d}})-P-$chaotic then it is $\F_t-$Devaney chaotic.
\end{theorem}
\begin{proof}
	Suppose $f$ is a chain mixing. If $f$ is $(\N_0,\F_{\underline{d}})-P-$chaotic then it is weakly mixing \cite{Oprocha2} and therefore $f$ is $\F_t-$Devaney chaotic.
\end{proof}

\begin{theorem}\label{3.7}
	Let $(X,f)$ be a compact dynamical system. Suppose $f$ has shadowing property. Then $f$ is $(\mathscr{D},\F_{\underline{d}})-P-$chaotic if and only if $f$ is $\F_{cf}-$Devaney chaotic.
\end{theorem}
\begin{proof}
	In the presence of shadowing property, $(\mathscr{D},\F_{\underline{d}})-$shadowing property is equivalent to topological mixing \cite{Oprocha2} and hence $(\mathscr{D},\F_{\underline{d}})-P-$chaos is equivalent to $\F_{cf}-$Devaney chaos by Theorem \ref{T35}.
\end{proof}

\begin{theorem}\label{3.8}
	Let $(X,f)$ be a compact dynamical system. Then following holds:
	\begin{enumerate}
		\item If $f$ is totally transitive $P-$chaotic map then $f$ is $\F_{cf}-$Devaney chaotic.
		\item If $X$ is connected and $f$ is non--wandering $P-$chaotic map then $f$ is $\F_{cf}-$Devaney chaotic. 
	\end{enumerate} 
\end{theorem}
\begin{proof}
	\begin{enumerate}
		\item In the presence of shadowing, totally transitive is equivalent to mixing \cite{TKS2}. Since $P(f)$ is dense in $X$, it follows that $f$ is $\F_{cf}-$Devaney chaotic. 
		\item Suppose $X$ is connected and $f$ is non--wandering map which has a shadowing property. Then $f$ is mixing \cite{TKS2} and hence $\F_{cf}-$Devaney chaotic.  
\end{enumerate} 	\end{proof}

\smallskip
\noindent In \cite{Arai}, Arai and Chinen proved that on continuum, $P-$chaos implies mixing and hence $\F_{cf}-$Devaney chaos. If we replace the $P-$chaos by $(\N_0,\F_{cf})-P-$chaos then using the same arguments given in the proof of Corollary 3.4 of \cite{Arai}, we get $f$ is mixing. Note that $(\N_0,\F_{t})-$shadowing is equivalent to $(\N_0,\F_{cf})-$shadowing \cite{Oprocha}. This gives us the following:

\begin{theorem} \label{C8}
	Every $(\N_0,\mathscr{F}_{t})-P-$chaotic map from a continuum to itself is $\mathscr{F}_{cf}-$Devaney chaotic.
\end{theorem}

\smallskip
\noindent In the following diagram we summarize the relationship between various types of $\F-$Devaney chaos and $(\F,\G)-P-$chaos. Note that we have used abbreviation DC for Devaney chaos and PC for $P-$chaos in the diagram. Also, the question mark on the arrows means we neither have a counter example to negate the arrow nor we have relation between the corresponding properties. For instance, we do not know under what conditions  $(\mathscr{D},\F_{\underline{d}})-P-$chaos will imply $(\mathscr{D},\mathscr{D})-P-$chaos nor we have counter example. 

\medskip 

\medskip 

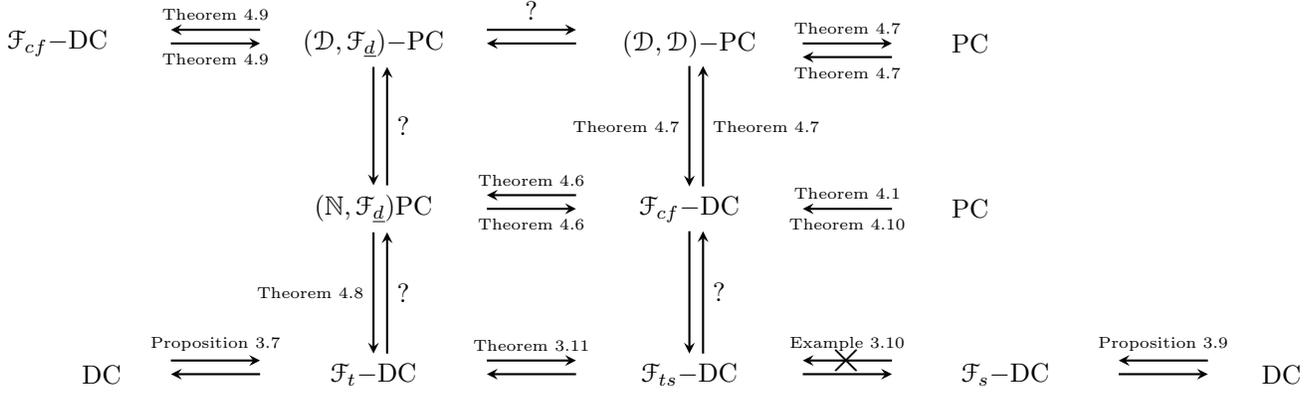
\tikzstyle{startstop} = [rectangle, rounded corners, minimum width=3cm, minimum height=0.5cm, text centered, draw=white]
\tikzstyle{arrow} = [thick, ->, >=stealth]
\begin{tikzpicture}[node distance=2.2cm]
	\hspace{-1.30cm}
	\node(DF)[startstop]{$(\mathscr{D},\F_{\underline{d}})-$PC};	
	\node(Fcf)[startstop, left of=DF, xshift=-2cm] {$\F_{cf}-$DC};
	\node(DD)[startstop, right of=DF, xshift=2cm] {$(\mathscr{D},\mathscr{D})-$PC};
	\node(PC)[startstop, right of=DD, xshift=2cm]{\hspace{-30pt} PC};		
	\node(NF)[startstop, below of=DF] {$(\N,\F_{\underline{d}})$PC};
	\node(Fcf2)[startstop, right of=NF, xshift=2cm] {$\F_{cf}-$DC};
	\node(PC2)[startstop, below of=PC]{\hspace{-30pt} PC};
	\node(Ft)[startstop, below of=NF] {$\F_t-$DC};
	\node(Fts)[startstop, right of=Ft, xshift=2cm] {$\F_{ts}-$DC};	
	\node(Fs)[startstop, right of=Fts,xshift=2cm] {$\F_s-$DC};
	\node(DC)[startstop, left of=Ft, xshift=-2cm]{\hspace{30pt} DC};
	\node(DC2)[startstop, right of=Fs, xshift=2cm] {\hspace{-30pt}DC};

	\draw [arrow] (DD) -- node[anchor=south] {\tiny Theorem \ref{3.5}}(PC);
	\draw [arrow, transform canvas={yshift=-1.2ex}] (PC) -- node[anchor=north] {\tiny Theorem \ref{3.5}}(DD);
	\draw [arrow] (DD) -- (DF);
	\draw [arrow, transform canvas={yshift=1.2ex}] (DF) -- node[anchor=south] { ?}(DD);
	\draw [arrow] (DD) -- node[anchor=east] {\tiny Theorem \ref{3.5}}(Fcf2);
	\draw [arrow, transform canvas={xshift=1.2ex}] (Fcf2) -- node[anchor=west] {\tiny Theorem \ref{3.5}}(DD);
	\draw [arrow] (DF) -- (NF);
	\draw [arrow, transform canvas={xshift=1.2ex}] (NF) -- node[anchor=west] { ?}(DF);
	\draw [arrow] (PC2) -- node[anchor=south] {\tiny Theorem \ref{3.1}} node[anchor=north]{\tiny Theorem \ref{3.8}}(Fcf2);
	\draw [arrow,  transform canvas={yshift=1.2ex}] (Fcf2) -- node[anchor=south] {\tiny Theorem \ref{3.4}}(NF);
	\draw [arrow] (NF) -- node[anchor=north] {\tiny Theorem \ref{3.4}}(Fcf2);
	\draw [arrow] (NF) -- node[anchor=east] {\tiny Theorem \ref{3.6}}(Ft);
	\draw [arrow, transform canvas={xshift=1.2ex}] (Ft) -- node[anchor=west] { ?}(NF);
	\draw [arrow] (Fcf2) -- (Fts);
	\draw [arrow, transform canvas={xshift=1.2ex}] (Fts) --  node[anchor=west] {?} (Fcf2);
	
	\draw [arrow] (Fts) -- (Ft);
	\draw [arrow, transform canvas={yshift=1.2ex}] (Ft) -- node[anchor=south] {\tiny Theorem \ref{3.10}}(Fts);
	
	\draw [arrow, transform canvas={yshift=1.2ex}] (DF) --node[anchor=south] {\tiny Theorem \ref{3.7}} (Fcf);
	\draw [arrow] (Fcf) -- node[anchor=north] {\tiny Theorem \ref{3.7}}(DF);
	\draw [arrow,  transform canvas={yshift=1.2ex}] (Fs) -- node[anchor=south]{\tiny Example \ref{2.6}}(Fts);
	\draw [arrow] (Fts) -- node[anchor=west, yshift=1.2ex]{\hspace{-0.5cm} \huge $\times$}(Fs);			
	\draw [arrow] (Ft) -- (DC);	
	\draw [arrow, transform canvas={yshift=1.2ex}] (DC) --node[anchor=south]{\tiny Proposition \ref{p3.7}} (Ft);																									
	\draw [arrow] (Fs) -- (DC2);
	\draw [arrow, transform canvas={yshift=1.2ex}] (DC2) --node[anchor=south]{\tiny Proposition \ref{3.9}} (Fs);																																
\end{tikzpicture}

\section{$\F-$Devaney chaos on interval maps}
\medskip
\noindent Recall, a map $f$ is said to be \emph{locally eventually onto (or leo)} if, for every non--empty open subset $U$ of $X$, there exists an integer $N$ such that $f^n(U)=X$ for all $n\geq N$. Note that locally eventually onto map is always a mixing.
\begin{example}
	Consider an interval map $f:[a,b]\longrightarrow [a,b]$ such that $f$ is a leo map. Then $f$ is mixing. Since on interval, transitivity implies $P(f)$ dense \cite{SR}, it follows that $P(f)$ is dense in $[a,b]$ and hence $f$ is $\F_{cf}-$Devaney chaotic by Theorem \ref{T35}.
\end{example}

\begin{lemma}\cite{SR}\label{L1}
	Consider a dynamical system $([0,1],f)$. Then the following are equivalent:
	\begin{enumerate}
		\item $f$ is totally transitive.
		\item $f$ is transitive with at least one periodic point of odd period different from $1$.
		\item $f$ is mixing.
		\item $f^n$ is mixing, for every $n\in \N$.
		
	\end{enumerate}
\end{lemma}

\begin{theorem}
	Consider dynamical system $([0,1],f)$. Then following holds:
	\begin{enumerate}
		\item $f$ is $\F_t-$Devaney chaotic then $f$ is $\F_{cf}-$Devaney chaotic.
		\item $f$ is Devaney chaotic with at least one periodic point of odd period different from $1$ if and only if $f$ is $\F_{cf}-$Devaney chaotic.
		\item $f$ is $\F_{cf}-$Devaney chaotic if and only if $f^n$ is $\F_{cf}-$Devaney chaotic for every $n\in \N$.
		\item If $f$ is a transitive map having at least two fixed points then it is $\F_{cf}-$Devaney chaotic.
	\end{enumerate}
\end{theorem}
\begin{proof}
	\begin{enumerate}
		\item Suppose $f$ is $\F_t-$Devaney chaotic. Therefore $f$ is weakly mixing on $[0,1]$. Hence by Lemma \ref{L1}, $f$ is mixing and therefore $f$ is $\F_{cf}-$Devaney chaotic.
		\item By Lemma \ref{L1}, $f$ is $f$ is mixing. Therefore, $\F_{cf}-$Devaney chaotic. 
		\item Suppose $f$ is $\F_{cf}-$Devaney chaotic. Therefore $f$ is mixing and hence $f^m$ is mixing. Again, $f$ is mixing implies $f^m$ is transitive, for each $m\in \N$. Therefore $P(f^m)$ is dense in $[0,1]$. Hence by Theorem \ref{T35}, $f^m$ is $\F_{cf}$ Devaney chaotic.
		
		\noindent Conversely, suppose $f^n$ is $\F_{cf}-$Devaney chaotic. Therefore $f$ is mixing. Hence by Lemma \ref{L1}, $f$ is $\F_{cf}-$Devaney chaotic.
		
		\item Suppose $f$ has at least two fixed points. Therefore $f$ is totally transitive. By Lemma \ref{L1}, $f$ is mixing and hence $\F_{cf}-$Devaney chaotic.
	\end{enumerate}
\end{proof}
	
	\bibliographystyle{amsplain}
	
\end{document}